\documentclass[12pt]{amsart}

\usepackage[draft]{todonotes}   

 \usepackage{amsmath}
\usepackage{amssymb}
\usepackage{amsthm}
\usepackage{graphicx}
\usepackage{amscd}
\usepackage{color}

\newcommand{\Proof}{\medskip\noindent{\bf Proof.}\quad}          

\newtheorem{theorem}{Theorem}
\newtheorem{lemma}[theorem]{Lemma}

\newtheorem{definition}{Definition}

\def\bF{{\mathbb{F}}}

\def\PG{{\mathrm{PG}}}

\def\cV{{\mathcal{V}}}
\def\cA{{\mathcal{A}}}

\def\cH{\mathcal{H}}

\title{Arcs and tensors}
\author{Simeon Ball and Michel Lavrauw}
\date{23 May 2019.\\ The first author acknowledges the support of the project MTM2017-82166-P of the Spanish {\em Ministerio de Econom\'ia y Competitividad.} The second author acknowledges the support of {\em The Scientific and Technological Research Council of Turkey}, T\"UB\.{I}TAK (project no. 118F159).\\
}                                           

\begin{document}
\begin{abstract}
To an arc $\cA$ of $\PG(k-1,q)$ of size $q+k-1-t$ we associate a tensor in $\langle \nu_{k,t}(\cA)\rangle^{\otimes k-1}$, where $\nu_{k,t}$ denotes the Veronese map of degree $t$ defined on $\PG(k-1,q)$.  As a corollary we prove that for each arc $\cA$ in $\PG(k-1,q)$ of size $q+k-1-t$, which is not contained in a hypersurface of degree $t$, there exists a polynomial $F(Y_1,\ldots,Y_{k-1})$ (in $k(k-1)$ variables) where $Y_j=(X_{j1},\ldots,X_{jk})$, which is homogeneous of degree $t$ in each of the $k$-tuples of variables $Y_j$, which upon evaluation at any $(k-2)$-subset $S$ of the arc $\cA$ gives a form of degree $t$ on $\PG(k-1,q)$ whose zero locus is the tangent hypersurface of $\cA$ at $S$, i.e. the union of the tangent hyperplanes of $\cA$ at $S$. This generalises the equivalent result for planar arcs ($k=3$), proven in \cite{BaLa2018}, to arcs in projective spaces of arbitrary dimension. A slightly weaker result is obtained for arcs in $\PG(k-1,q)$ of size $q+k-1-t$ which are contained in a hypersurface of degree $t$. We also include a new proof of the Segre-Blokhuis-Bruen-Thas hypersurface associated to an arc of hyperplanes in $\PG(k-1,q)$.
\end{abstract}
\maketitle

\section{Introduction and motivation}
An {\it arc} of $\PG(k-1,q)$ is a set of points no $k$ of which are contained in a hyperplane. Arcs are the subject of Segre's fundamental problems proposed in 1955 \cite{Segre1955b} and they play an important role in Galois geometry \cite{Segre1967}. Segre's celebrated result from \cite{Segre1955a} which says that an arc of size $q+1$ in $\PG(2,q)$, $q$ odd, is a conic, has inspired many mathematicians to work on problems related to arcs in projective spaces over finite fields.
Normal rational curves are well known examples of arcs of size $q+1$. There are arcs of size $q+2$ in $\PG(2,q)$ when $q$ is even called {\it hyperovals}. For a list of the collineation groups of these arcs, see \cite{PP1999}.

Another driving force for the study of arcs is the fact that they are equivalent to linear Maximum Distance Separable codes (MDS codes), which according to \cite{MS1977} form ``one of the most fascinating chapters in all of coding theory''. 
These codes have been extensively studied and a well-known conjecture (called the {\it MDS conjecture}) claims that if $4 \leq k \leq q-2$, then a $k$-dimensional linear MDS code over the finite field with $q$ elements has length at most $q+1$.  The MDS conjecture was proven for $q$ prime in \cite{Ball2012}.

The most recent result from \cite[Corollary 4]{BaLa2018} verifies the MDS conjecture for $k \leq \sqrt{q}-\sqrt{q}/p+2$, in the case that $q=p^{2h}$ and $p$ is an odd prime.
Contrary to most previous results in this direction (for example, the bounds from \cite{HK1996}, \cite{HK1998}, \cite{Voloch1987}, \cite{Voloch1990a}, \cite{Voloch1990b} and \cite{Voloch1991}) the result from \cite{BaLa2018} does not rely on Segre's algebraic envelope associated to an arc, and deep results on the number of points on algebraic curves over finite fields, in particular the Hasse-Weil theorem and the St\"ohr-Voloch theorem. 
Instead, the results in \cite{BaLa2018} are based on the existence of a certain bi-homogeneous polynomial which upon evaluation at a point of the arc splits into linear factors corresponding to the tangents of the arc through that point. In this paper, this is generalised to arcs in projective spaces of arbitrary dimension, resulting in Theorem \ref{thm:form_existence}, which is proved in Section 5.

In Section~\ref{BBThyp} we compare this result to the hypersurface associated to an arc of hyperplanes as obtained in the sequence of papers \cite{Segre1967} for $k=3$, in \cite{BBT1988} for $k=4,5$, and \cite{BBT1990} for arbitrary dimension $k\geq 3$.

\section{The tangent hypersurfaces and the main theorem}

Throughout, $\cA$ will be an arc of $\PG(k-1,q)$ of size $q+k-1-t$, arbitrarily ordered, and we identify each point of $\cA$ with a fixed vector representative. 
Let $V_r[X]$ denote the vector space of forms (homogeneous polynomials) of degree $r$ in $\bF_q[X_1,\ldots,X_k]$, and $\Phi_r[X]$ the subspace of $V_r[X]$ consisting of forms vanishing on $\cA$. As in the previous sentence we will often write $X$ instead of $X_1,\ldots, X_k$. 

Each subset $S$ of size $k-2$ of $\cA$ is contained in precisely $t$ hyperplanes of $\mathrm{PG}(k-1,q)$ meeting $\cA$ exactly in $S$ (called {\em tangent $S$-hyperplanes}). Their union forms the {\em tangent hypersurface of $\cA$ at $S$}. Each such hypersurface has degree $t$ and is the zero locus of
\begin{eqnarray}\label{eqn:f_S}
f_S(X)=\prod_{i=1}^t \alpha_i(X),
\end{eqnarray}
where $\alpha_i(X)$, $i=1,\ldots,t$, are linear forms whose kernels are the $t$ tangent $S$-hyperplanes. This defines $f_S(X)$ up to a nonzero scalar factor, which we will now determine based on the evaluation of $f_S$ at carefully chosen points of $\cA$.

Let $E$ be the set of the first $k-2$ elements of $\cA$. For each $(k-2)$-subset $S\subset \cA$,
scale the polynomial $f_S(X)$ so that
\begin{eqnarray}\label{eqn:scaling}
f_S(e)=(-1)^{s(t+1)}f_{S\cup\{e\}\setminus \{a\}}(a),
\end{eqnarray}
where $e$ is the first element of $E\setminus S$, $a$ is the last element of $S\setminus E$, and $s$ is the parity of the permutation which orders $S\cup \{e\}$ as in the ordering of $\cA$ (to determine the value of $s$ we assume the ordering of $\cA$ for the subset $S$). With this notation it should be understood that  the order is respected when taking the union of ordered sets, i.e. with ``union" we mean the concatenation of the ordered sets.

We are now in a position to state the main result of this article.

\begin{theorem}\label{thm:form_existence}
Let $\cA$ be an arc in $\PG(k-1,q)$ of size $q+k-1-t$ and let $\Phi_t[X]$ denote the space of homogeneous polynomials of degree $t$ in $X=(X_1,\ldots,X_k)$ which are zero on $\cA$. There exists
a homogeneous polynomial $F(Y_1,\ldots,Y_{k-1})$ (in $k(k-1)$ variables) where $Y_j=(X_{j1},\ldots,X_{jk})$, and $F$ is homogeneous of degree $t$ in each of the $k$-tuples of variables $Y_j$, with the following properties.

\begin{enumerate}
\item[(i)]
For every $(k-2)$-subset $S=[a_1,\ldots,a_{k-2}]$ of the arc $\cA$ we have 
$$
F(a_1,\ldots,a_{k-2},X)=(-1)^{s(t+1)}f_S(X) \mbox{ modulo } \Phi_t[X],
$$
where $s$ is the parity of the permutation which orders $S$ as in the ordering of $\cA$.
\item[(ii)]
For every sequence $a_1,\ldots,a_{k-1}$ of elements of $\cA$ in which points are repeated, 
$$
F(a_1,\ldots,a_{k-1})=0.
$$
\item[(iii)]
For every permutation $\sigma \in \mathrm{Sym}(k-1)$,
$$
F(Y_{\sigma(1)},\ldots,Y_{\sigma(k-1)})=(-1)^{s(t+1)}F(Y_1,\ldots,Y_{k-1}),
$$
modulo $\Phi_t[Y_1],\ldots,\Phi_t[Y_{k-1}]$, where $s$ is the parity of $\sigma$.
\item[(iv)]
Any form $F(Y_1,\ldots,Y_{k-1})$ satisfying {\rm(i), (ii)} and {\rm (iii)} is unique modulo $\Phi_t[Y_1]$, $\ldots,$ $\Phi_t[Y_{k-1}]$.
\end{enumerate}
\end{theorem}

The following three sections are mainly dedicated to proving Theorem~\ref{thm:form_existence}.

\section{The scaled coordinate-free lemma of tangents}
In this section we prove what we call the {\em scaled coordinate-free lemma of tangents} for an arc in a projective space of arbitrary dimension. The original lemma of tangents is due to Segre \cite{Segre1967}. A coordinate-free version was given in \cite{Ball2012}, and a scaled coordinate-free version for the planar case was introduced in \cite{BaLa2018}.

As before, $\cA$ is an arc in $\PG(k-1,q)$, with tangent hypersurfaces given as the zero loci of the forms $f_S(X)$ as defined in (\ref{eqn:f_S}) and scaled as in (\ref{eqn:scaling}).
Define the function $g$ on ordered subsets of $\cA$ of size $k-1$ by
\begin{eqnarray}\label{eqn:g}
g(S\cup\{a\})= (-1)^{s(t+1)}f_S(a),
\end{eqnarray} 
where $S$ is an ordered subset of $\cA$ of size $k-2$ and $s$ is the parity of the permutation which orders $S$ as in the ordering of $\cA$. Note that $S$ is considered as an unordered set in the notation $f_S(a)$.
Extend the definition of $g$ by setting it equal to zero when evaluated at $(k-1)$-tuples with repeated elements. Recall that $E$ consists of the first $k-2$ elements of $\cA$.

\begin{lemma}\label{lem:g_functionpre}
If $\sigma$ is a permutation in $\mathrm{Sym}(k-1)$ and $T$ is an ordered $(k-1)$-subset of $\cA$ containing $E$, then
$$
g(T^\sigma)=(-1)^{s(t+1)}g(T),
$$
where $s$ is the parity of the permutation $\sigma$.
\end{lemma}
\Proof
If $\sigma$ is a permutation in $\mathrm{Sym}(k-1)$ fixing $k-1$ then, by definition, 
\begin{eqnarray}\label{eqn:sigma}
g(T^\sigma)=(-1)^{s(t+1)}g(T),
\end{eqnarray}
where $s$ is the parity of the permutation $\sigma$.

So in order to prove the assertion it suffices to show that 
$$
g(a_1,\ldots,a_{k-3},a_{k-1},a_{k-2})=(-1)^{t+1}g(a_1,\ldots,a_{k-1}),
$$ 
for any distinct $a_1,\ldots,a_{k-1}\in \cA$.

Consider any $k-1$ distinct points $a_1,\ldots,a_{k-1}\in \cA$, and put $T=[a_1,\ldots,a_{k-1}]$. Let $T_j$ denote the ordered set obtained from $T$ by removing the $j$-th point.
If $E=[e_1,\ldots,e_{k-2}]=[a_1,\ldots,a_{k-2}]$ then by the definition of $g$ and the scaling (\ref{eqn:scaling}) of the tangent forms $f_S(X)$ we have
$$
g(a_1,\ldots,a_{k-1})=g(e_1,\ldots,e_{k-2},a_{k-1})=f_E(a_{k-1})=(-1)^{(k-2)(t+1)}f_{T_1}(e_1).
$$
This is equal to 
$$
(-1)^{(k-2)(t+1)}g(e_2,\ldots,e_{k-2},a_{k-1},e_1)=(-1)^{(t+1)}g(a_{k-1},e_2,\ldots,e_{k-2},e_1)
$$
where the last equality was obtained by applying (\ref{eqn:sigma}).

Likewise, for any $j\in \{2,\ldots,k-1\}$ we obtain 
$$
f_E(a_{k-1})=(-1)^{(k-1-j)(t+1)}f_{T_j}(e_j) 
$$
which is equal to 
$$(-1)^{(k-1-j)(t+1)}g(e_1,\ldots,e_{j-1},e_{j+1},\ldots,e_{k-2},a_{k-1},e_j),$$
and by applying  (\ref{eqn:sigma}) we obtain
$$
g(e_1,\ldots,e_{k-2},a_{k-1})=(-1)^{(t+1)}g(e_1,\ldots,e_{j-1},a_{k-1},e_{j+1},\ldots,e_{k-2},e_j).
$$
We have shown that for any $T=[e_1,\ldots,e_{k-2},a_{k-1}]$,
$$g(T^\sigma)=(-1)^{t+1}g(T)$$
for any transposition $\sigma=(j,k-1)$. In combination with (\ref{eqn:sigma}) this proves the lemma.
\qed

Next we formulate and prove the main result of this section.

\begin{lemma}\label{lem:g_function}{\bf[Scaled coordinate-free lemma of tangents]}
Let $\cA$ be an arc in $\PG(k-1,q)$, with tangent hypersurfaces given as the zero loci of the forms $f_S(X)$ as defined in (\ref{eqn:f_S}) and scaled as in (\ref{eqn:scaling}), and let $g$ be the function as defined in (\ref{eqn:g}). 
If $\sigma$ is a permutation in $\mathrm{Sym}(k-1)$ and $T$ is a $(k-1)$-subset of $\cA$ then
$$
g(T^\sigma)=(-1)^{s(t+1)}g(T),
$$
where $s$ is the parity of the permutation $\sigma$.
\end{lemma}
\Proof
Pick any ordered subset $B=[a_1,\ldots,a_k]\subset \cA$ of size $k$. Since $\cA$ is an arc, it follows that $B$ is a basis. 
Denote by $B_{l,i,j}$ the ordered set obtained from $B$ by removing the $l$-th, the $i$-th and the $j$-th point from $B$ and by $B_{l,i,j}(x,y)$ the ordered set of points obtained from $B$ by removing the $l$-th point from $B$ and replacing the $i$-th point by $x$ and the $j$-th point by $y$. 

Let $1\leq l<j<k$ be fixed. For $x,y\in \cA\setminus B_{l,j,k}$ define
$$h(x,y)=g(B_{l,j,k}(x,y)).$$
Then for any point $u$, with coordinates $(u_1,\ldots,u_k)$ w.r.t. $B$, we have
$$
h(a_{l},u)=\prod_{i=1}^t (b_{ij}u_{j}+b_{ik}u_k),\quad
h(a_{j},u)=\prod_{i=1}^t (c_{il}u_{l}+c_{ik}u_k),\quad
h(a_{k},u)=\prod_{i=1}^t (d_{il}u_{l}+d_{ij}u_{j}),
$$
for some $b_{ij}, c_{ij}, d_{ij} \in {\mathbb F}_q$.

Let $B_{l,j}$ denote the ordered set of points obtained from $B$ by removing the $l$-th and the $j$-th point. With respect to the basis $B$, the hyperplane containing $\langle B_{l,j}\rangle$ and $s=(s_1,s_2,\ldots,s_k) \in \cA \setminus B$ is the kernel of the linear form $X_{l}-(s_{l}/s_j)X_j$. Since these hyperplanes are distinct from the tangent $(B_{l,j})$-hyperplanes, together they constitute all hyperplanes containing $B_{l,j}$, except the kernels of the linear forms $X_{l}$ and $X_j$. Hence, 
$$
\prod_{i=1}^t \frac{d_{ij}}{d_{il}} \prod_{s \in \cA \setminus B} \frac{-s_{l}}{s_j}=\prod_{d \in {\mathbb F}_q \setminus \{ 0\}} d = -1.
$$
Observing that $h(a_k,a_j)=\prod_{i=1}^t d_{ij}$ and $h(a_{k},a_l)=\prod_{i=1}^t d_{il}$, this gives
$$
h(a_{k},a_j) \prod_{s \in \cA \setminus B} s_{l}= (-1)^{|\cA\setminus B|+1} h(a_{k},a_{l}) \prod_{s \in \cA \setminus B} s_{j}.
$$
Similarly, by considering hyperplanes through $B_{l,k}$,
$$
h(a_j,a_l) \prod_{s \in \cA \setminus B} s_{k}= (-1)^{|\cA\setminus B|+1} h(a_j,a_k) \prod_{s \in \cA \setminus B} s_{l}.
$$
and by considering hyperplanes through $B_{j,k}$,
$$
h(a_l,a_k) \prod_{s \in \cA \setminus B} s_{j}= (-1)^{|\cA\setminus B|+1} h(a_l,a_j) \prod_{s \in \cA \setminus B} s_{k}.
$$
Combining these three equations, and observing that $(-1)^{|\cA\setminus B|+1}=(-1)^{t+1}$, we obtain
$$h(a_j,a_l)=(-1)^{t+1}\quad h(a_l,a_j)
\quad \frac{h(a_j,a_k)h(a_k,a_l)}{h(a_k,a_j)h(a_l,a_k)}.
$$
We can rewrite this as
\begin{eqnarray}\label{eqn:transpositions}
g(B_{k}^{(jl)})=(-1)^{t+1}g(B_{k})
\frac{g(B_l)g(B_{j}^{(lk)})}{g(B_{l}^{(jk)})g(B_j)}
\end{eqnarray}
where $B_l$ is obtained from $B$ by removing the $l$-th vector, and with the understanding that $B_l^\sigma$ denotes the result of
 removing the $l$-th vector from $B$ after applying the permutation $\sigma\in {\mathrm{Sym}}(k)$ to the $k$ positions.

We will prove the lemma by induction on the size of $T\setminus E$ (as sets), where as before $E$ consists of the first $k-2$ elements of the ordered arc $\cA$.

If $|T\setminus E|=1$ then the lemma follows from Lemma~\ref{lem:g_functionpre}.

Suppose that for each ordered $(k-1)$-tuple $T$ for which $T\setminus E$ has size at most
$r\geq 1$, we have 
$$g(T^\sigma)=(-1)^{(t+1)}g(T)$$
for any transposition $\sigma \in \mathrm{Sym}(k-1)$.

Consider an ordered $(k-1)$-tuple $T=[a_1,\ldots,a_{k-1}]$ with $T\setminus E$ of size $r+1$. 

Suppose that $a_{k-2}, a_{k-1} \not\in E$. 
Let 
$e_\eta$ denote the first point in $E\setminus T$ in the ordering of $\cA$, and put $B=[a_1,\ldots, a_{k-1},e_\eta]$.
Then the left hand side of (\ref{eqn:transpositions}),
with $j=k-2$ and $l=k-1$, becomes
$$
g(B_{k}^{(jl)})=g(a_1,\ldots,a_{k-3},a_{k-1},a_{k-2})
$$ 
while the right hand side equals
$$
(-1)^{t+1}g(B_k)
\frac{g(a_1,\ldots,a_{k-3},a_{k-2},e_\eta)g(a_1,\ldots,a_{k-3},e_\eta,a_{k-1})}{g(a_1,\ldots,a_{k-3},e_\eta,a_{k-2})g(a_1,\ldots,a_{k-3},a_{k-1},e_\eta)},
$$
which by the induction hypothesis equals 
$$(-1)^{t+1}g(B_k)=(-1)^{t+1}g(a_1,\ldots,a_{k-3},a_{k-2},a_{k-1}),
$$
since $B_l\setminus E$ and $B_j\setminus E$ are of size $r$. 
This proves that if the points of $T$ in position $k-2$ and $k-1$ do not belong to $E$ then 
\begin{eqnarray}\label{eqn:step1}
g(T^\sigma)=(-1)^{t+1}g(T)
\end{eqnarray} 
for the transposition $\sigma=(k-2,k-1)$.

Next, suppose $a_{k-2}\in E$ and $a_{k-1}\notin E$ is the last point of $T$ in the ordering of $\cA$.
Let $e_\eta$ denote the first point in $E\setminus (T \setminus \{a_{k-2}\})$ in the ordering of $\cA$. 

Let $S=\{a_1,\ldots,a_{k-3},a_{k-1}\}$. By the scaling (\ref{eqn:scaling}) of the tangent forms $f_S(X)$ we have
\begin{eqnarray}\label{eqn:intermediate2}
f_S(e_\eta)=(-1)^{s(t+1)}f_{S\setminus\{a_{k-1}\}\cup \{e_\eta\}}(a_{k-1}),
\end{eqnarray}
where $s$ is the number of transpositions needed to reorder $S\cup\{e_\eta\}$ as in the ordering of $\cA$.
Moreover, by the definition of $g$, we have
$$
f_{S}(e_\eta)=(-1)^{s_1(t+1)}g(a_1,\ldots,a_{k-3},a_{k-1},e_\eta),
$$
where $s_1$ is the number of transpositions needed to reorder $[a_1,\ldots,a_{k-3},a_{k-1}]$ as in the ordering of $\cA$,
and 
$$
f_{S\setminus\{a_{k-1}\}\cup \{e_\eta\}}(a_{k-1})=(-1)^{s_2(t+1)}g(a_1,\ldots,a_{k-3},e_\eta,a_{k-1})
$$
where $s_2$ is the number of transpositions needed to reorder $[a_1,\ldots,a_{k-3},e_\eta]$ as in the ordering of $\cA$.
Since $a_{k-1}$ is the last point of $T$ in the ordering of $\cA$ we have 
$$s_2\equiv s_1+s-1\mod 2$$
and therefore 
$$
(-1)^{s_1(t+1)}(-1)^{s(t+1)}(-1)^{s_2(t+1)}=(-1)^{(t+1)}.
$$
Combining this with (\ref{eqn:intermediate2}) we obtain
\begin{eqnarray}\label{eqn:finale}
g(a_1,\ldots,a_{k-3},a_{k-1},e_\eta)=(-1)^{(t+1)}g(a_1,\ldots,a_{k-3},e_\eta,a_{k-1}).
\end{eqnarray}

Let $B$ denote the ordered $k$-tuple obtained from $T$ by adding the point $e_\eta\in E$. With $j=k-2$ and $l=k-1$, the induction hypothesis implies
$$g(B_l)=(-1)^{(t+1)}g(B_l^{(jk)})$$
since $B_l\setminus E$ has size $r$, and by (\ref{eqn:finale})
$$g(B_j)=(-1)^{(t+1)}g(B_j^{(lk)}).$$

Therefore by 
(\ref{eqn:transpositions}) we obtain $g(B_k^{(jl)})=(-1)^{(t+1)}g(B_k)$, i.e.
$$g(T)=(-1)^{(t+1)}g(T^{\sigma}),$$
for the transposition $\sigma=(k-2,k-1)$.

Next suppose $a_{k-2}\in E$, $a_{k-1} \not\in E$ and $a_{k-1}$ is not the last point of $T$, in the ordering of $\cA$.
If $a_j$ is the last point of $T$ in the ordering of $\cA$, then $a_j\notin E$ and therefore $j<k-2$.
Consider the transpositions $\tau=(j,k-2)$ and $\sigma=(k-2,k-1)$. 
Applying the permutation $\tau\sigma\tau\sigma\tau$ to $T$ we get
$$
T^{\tau\sigma\tau\sigma\tau}=[a_1,\ldots,a_j,\ldots, a_{k-1},a_{k-2}]
$$
which is $T^\sigma$. Moreover, the first time that $\sigma$ is applied, the pair of points
in the last two positions is $(a_j,a_{k-1})$, consisting of two points of $T\setminus E$, and therefore by (\ref{eqn:step1}) this
gives a factor $(-1)^{(t+1)}$ to the evaluation of $g$. The second time $\sigma$ is applied, the pair of points in the last
two positions is $(a_{k-2},a_j)$ where $a_{k-2}\in E$ and $a_j$ is the last point of $T$ in the ordering of 
$\cA$, and so, this time by (\ref{eqn:intermediate2}), this gives a factor $(-1)^{(t+1)}$ to the evaluation of $g$.

Finally, by (\ref{eqn:sigma}), each of the three applications of $\tau$ also gives a factor $(-1)^{(t+1)}$.
This amounts to a total of five factors $(-1)^{(t+1)}$, and we may conclude that also in this case
$$g(T)=(-1)^{(t+1)}g(T^{\sigma}),$$
for the transposition $\sigma=(k-2,k-1)$.

Thus, we have proved that if $a_{k-2} \in E$ and $a_{k-1} \not\in E$ or if $a_{k-2} \not\in E$ and $a_{k-1} \in E$ then 
\begin{eqnarray}\label{eqn:step2}
g(T^\sigma)=(-1)^{t+1}g(T)
\end{eqnarray} 
for the transposition $\sigma=(k-2,k-1)$. 

\bigskip

Finally suppose that both elements $a_{k-2}, a_{k-1} \in E$. 
Let $a_j$ ($j\in \{1,\ldots,k-3\}$) be a point of $T\setminus E$ and consider the transpositions
$\tau=(j,k-2)$ and $\sigma=(k-2,k-1)$. Then, similarly as above we have
$T^{\tau\sigma\tau\sigma\tau}=T^\sigma$, which this time also making use of
(\ref{eqn:step2}) implies
\begin{eqnarray}\label{eqn:step3}
g(T^{\sigma})=(-1)^{t+1}g(T).
\end{eqnarray} 
This concludes the proof.\qed

\section{A tensor associated with an arc}
In this section we show how the coordinate-free lemma of tangents can be used to construct a particular tensor which will eventually lead to our main result Theorem \ref{thm:form_existence}.
Let $\nu_{k,t}$ denote the {\it Veronese map of degree $t$} 
$$\nu_{k,t}~:~\PG(k-1,q) \rightarrow \PG(N-1,q)~:~x=(x_1,\ldots,x_k)\mapsto (\ldots, x^I,\ldots),$$
where $N={{k+t-1}\choose{t}}$. Under this map, the image of a point $x$ is the point whose coordinate vector consists of all possible monomials $x^I$ of degree $t$ in $x_1,\ldots, x_k$. Thus, a coordinate of the image of $x$ is of the form $x^I=x_1^{d_1}\cdots x_k^{d_k}$, where $d_1+\cdots+d_k=t$. The image of the Veronese map is an algebraic variety, called the  {\it Veronese variety}, and is denoted by $\cV_{k,t}(\bF_q)$.

For each ordered $(k-2)$-subset $S\subset \cA$ we consider the associated linear form $h_S\in \bF_q[Z_1,\ldots,Z_N]$ defined by
\begin{eqnarray}
h_S\circ \nu_{k,t} =f_S.
\end{eqnarray}


We define a function $h$ from
$$\nu_{k,t}(\cA)\times \nu_{k,t}(\cA)\times \ldots \times \nu_{k,t}(\cA) \quad \mbox{  ($k-1$ factors)}$$ to ${\mathbb F}_q$
by
\begin{eqnarray}
h(\nu_{k,t}(a_1),\nu_{k,t}(a_2),\ldots,\nu_{k,t}(a_{k-1})):=g( a_1,a_2,\ldots, a_{k-1}).
\end{eqnarray}
A {\em $t$-socle} for $\cA$ is a set of points of $\cA$ whose image under the Veronese map of degree $t$ spans the subspace spanned by $\cA$ under the Veronese map. So a $t$-socle is a set of points $e_1,\ldots,e_m\in \cA$ for which 
$$\langle \nu_{k,t}(e_1),\ldots,\nu_{k,t}(e_m)\rangle=\langle \nu_{k,t}(\cA)\rangle.
$$
We define the function $\bar h$ from $\langle \nu_{k,t}(\cA)\rangle^{\otimes k-1}$ to ${\mathbb F}_q$ by
$$\bar h\left ( \sum_{i_1} c_{1i_1} \nu_{k,t}(e_{i_1})\otimes \sum_{i_2} c_{2i_2} \nu_{k,t}(e_{i_2})\otimes \ldots\otimes \sum_{i_{k-1}} c_{1i_{k-1}} \nu_{k,t}(e_{i_{k-1}})\right )
$$
$$:=\sum_{i_1} c_{1i_1}\sum_{i_2} c_{2i_2} \ldots \sum_{i_{k-1}} c_{1i_{k-1}}  
g(e_{i_1},\ldots,e_{i_{k-1}}),
$$
where each sum is from $i_j=1,\ldots,m$.

We will show that for each $a_1, \ldots, a_{k-1} \in \cA$,
$$
\bar h(\nu_{k,t}(a_1)\otimes \nu_{k,t}(a_2)\otimes \ldots \otimes \nu_{k,t}(a_{k-1}))=g( a_1,a_2,\ldots, a_{k-1}).
$$

\begin{lemma}\label{lem:bar_h}
The function $\bar h$ defines a multilinear form on $\langle \nu_{k,t}(\cA)\rangle^{\otimes k-1}$ whose restriction to
$$\nu_{k,t}(\cA)\times \nu_{k,t}(\cA)\times \ldots \times \nu_{k,t}(\cA) \quad \mbox{  ($k-1$ factors)}$$
equals $h$.
\end{lemma}
\Proof
By definition, the function $\bar h$ is multilinear, and coincides with $h$ when evaluated at arguments of the form
$${\bf v}=\left ( \nu_{k,t}(e_{i_1}), \nu_{k,t}(e_{i_2}),\ldots, \nu_{k,t}(e_{i_{k-1}})\right ).$$
For each $x \in \cA$ with $\nu_{k,t}(x)=\sum_i \lambda_i \nu_{k,t}(e_i)$, and for each $j\in \{1,\ldots , k-2\}$, we have
$$
h\left ( \nu_{k,t}(e_{i_1}), \ldots, \nu_{k,t}(e_{i_{j-1}}),   \nu_{k,t}(x),\nu_{k,t}(e_{i_{j+1}}), \ldots, \nu_{k,t}(e_{i_{k-1}})\right )
$$
$$
=g(e_{i_1},\ldots,e_{i_{j-1}},x,e_{i_{j+1}},\ldots, e_{i_{k-2}},e_{i_{k-1}})
$$
$$
=(-1)^{t+1}g(e_{i_1},\ldots,e_{i_{j-1}},e_{i_{k-1}},e_{i_{j+1}},\ldots,e_{i_{k-2}},x )
$$
$$=(-1)^{t+1} h_E(\nu_{k,t}(x))
$$
$$
=(-1)^{t+1} \sum_i \lambda_i h_E(\nu_{k,t}(e_i))
$$
where $E=[e_{i_1},\ldots,e_{i_{j-1}},e_{i_{k-1}},e_{i_{j+1}},\ldots,e_{i_{k-2}}]$. This in turn equals
$$
\sum_i \lambda_i  (-1)^{t+1}  g(e_{i_1},\ldots,e_{i_{j-1}},e_{i_{k-1}},e_{i_{j+1}},\ldots,e_{i_{k-2}},e_i )
$$
$$
=  \sum_i \lambda_i g(e_{i_1},\ldots,e_{i_{j-1}}, e_i,e_{i_{j+1}},\ldots,e_{i_{k-2}},e_{i_{k-1}})
$$
$$
=\sum_i \lambda_i \bar h\left ( \nu_{k,t}(e_{i_1})\otimes \ldots \otimes \nu_{k,t}(e_{i_{j-1}})\otimes   \nu_{k,t}(e_i)\otimes\nu_{k,t}(e_{i_{j+1}})\otimes \ldots \otimes \nu_{k,t}(e_{i_{k-1}})\right )
$$
$$
=\bar h\left ( \nu_{k,t}(e_{i_1}) \otimes \ldots \otimes \nu_{k,t}(e_{i_{j-1}}) \otimes   \nu_{k,t}(x) \otimes \nu_{k,t}(e_{i_{j+1}}) \otimes \ldots \otimes \nu_{k,t}(e_{i_{k-1}})\right ).
$$
This shows that $\bar h$ and $h$ are equal when evaluated at arguments obtained from 
$${\bf v}=\left ( \nu_{k,t}(e_{i_1}), \nu_{k,t}(e_{i_2}),\ldots, \nu_{k,t}(e_{i_{k-1}})\right ),$$
by replacing the $j$-th argument in  $\bf v$ by $\nu_{k,t}(x)$ ($x\in \cA$). 

The proof can now be finished by induction. As induction hypothesis we assume that the values of $\bar h$ and $h$ are equal when evaluated at $(k-1)$-tuples obtained from $\bf v$ by replacing $s\geq 1$
of the arguments of ${\bf v}$ by $\nu_{k,t}(x_1), \ldots, \nu_{k,t}(x_s)$ for any $s$ points $x_1, \ldots, x_s \in \cA$.

Let $\bf w$ be obtained from $\bf v$, by replacing $s+1$
of the arguments of ${\bf v}$ by expressions of the form $\nu_{k,t}(x_1), \ldots, \nu_{k,t}(x_{s+1})$ where $x_1, \ldots, x_{s+1} \in \cA$.

If $\nu_{k,t}(x_{s+1})$ is not in the last position of $\bf w$, then define $\bf w'$ as the $(k-1)$-tuple obtained from $\bf w$ by interchanging the argument where $x_{s+1}$ appears with the argument in the last position.
Then, by Lemma~\ref{lem:g_function},
$$h({\bf w})=(-1)^{t+1}h({\bf w}').$$

If $\nu_{k,t}(x_{s+1})$ is in the last position of $\bf w$ then put $\bf w'=w$.

Then $h({\bf w'})=h_E(\nu_{k,t}(x_{s+1}))$ for a suitable $E$, and since $h_E$ is a linear form, we can rewrite $h({\bf w'})$ as a linear combination of evaluations of $h$ at $(k-1)$-tuples obtained from $\bf v$ by replacing $s$ arguments of $\bf v$ by expressions of the form $\nu_{k,t}(x_1), \ldots, \nu_{k,t}(x_{s})$ with $x_1, \ldots, x_{s} \in \cA$. By induction the values of $\bar h$ and $h$ are equal when evaluated at such $(k-1)$-tuples.
\qed

\section{Proof of Theorem~\ref{thm:form_existence}}
The previous sections contain the necessary lemma's to prove the main theorem. We restate the theorem for the convenience of the reader.

{\bf Theorem~\ref{thm:form_existence}}
{\em 
Let $\cA$ be an arc in $\PG(k-1,q)$ of size $q+k-1-t$ and let $\Phi_t[X]$ denote the space of homogeneous polynomials of degree $t$ in $X=(X_1,\ldots,X_k)$ which are zero on $\cA$. There exists
a homogeneous polynomial $F(Y_1,\ldots,Y_{k-1})$ (in $k(k-1)$ variables) where $Y_j=(Y_{j1},\ldots,Y_{jk})$, and $F$ is homogeneous of degree $t$ in each of the $k$-tuples of variables $Y_j$, with the following properties.

\begin{enumerate}
\item[(i)]
For every $(k-2)$-subset $S=[a_1,\ldots,a_{k-2}]$ of the arc $\cA$ we have 
$$
F(a_1,\ldots,a_{k-2},X)=(-1)^{s(t+1)}f_S(X) \mbox{ modulo } \Phi_t[X],
$$
where $s$ is the parity of the permutation which orders $S$ as in the ordering of $\cA$.
\item[(ii)]
For every sequence $a_1,\ldots,a_{k-1}$ of elements of $\cA$ in which points are repeated, 
$$
F(a_1,\ldots,a_{k-1})=0.
$$
\item[(iii)]
For every permutation $\sigma \in \mathrm{Sym}(k-1)$,
$$
F(Y_{\sigma(1)},\ldots,Y_{\sigma(k-1)})=(-1)^{s(t+1)}F(Y_1,\ldots,Y_{k-1}),
$$
modulo $\Phi_t[Y_1],\ldots,\Phi_t[Y_{k-1}]$, where $s$ is the parity of $\sigma$.
\item[(iv)]
Any form $F(Y_1,\ldots,Y_{k-1})$ satisfying {\rm(i), (ii)} and {\rm (iii)} is unique modulo $\Phi_t[Y_1]$, $\ldots,$ $\Phi_t[Y_{k-1}]$.
\end{enumerate}
}
\Proof

Let $\cA$ be an arc of size $q+k-1-t$ in $\PG(k-1,q)$. By Lemma \ref{lem:bar_h}, there exists a multilinear form
$\bar h$ on $\langle \nu_{k,t}(\cA)\rangle^{\otimes k-1}$, such that for all $a_1,\ldots,a_{k-1}\in \nu_{k,t}(\cA)$
$$
{\bar h}(\nu_{k,t}(a_1),\ldots, \nu_{k,t}(a_{k-1}))=g(a_1,\ldots,a_{k-1})=(-1)^{s(t+1)}f_S(a_{k-1}),
$$
where $S=[a_1,\ldots,a_{k-2}]$ is an ordered subset of $\cA$ and $s$ is the parity of the permutation which orders $S$ as in the ordering of $\cA$.

The multilinear form $\bar h$ corresponds to a
hyperplane $\bar \cH$ in $\langle \nu_{k,t}(\cA)\rangle^{\otimes k-1}$. 
Let $\cH$ be a hyperplane of $\langle \cV_{k,t}(\bF_q)\rangle^{\otimes k-1}$ intersecting
$\langle \nu_{k,t}(\cA)\rangle^{\otimes k-1}$ in $\bar \cH$. 

The hyperplane $\cH$ is the zero locus of a linear form $\alpha$ on $\langle \cV_{k,t}(\bF_q)\rangle^{\otimes k-1}$.
This defines $\alpha$ up to a nonzero scalar factor. Now scale $\alpha$ such that the restriction of $\alpha$ to 
$\langle \nu_{k,t}(\cA)\rangle^{\otimes k-1}$ coincides with $\overline h$ (which is possible since 
 $\cH\cap \langle \nu_{k,t}(\cA)\rangle^{\otimes k-1}=\bar \cH$).

Denote by $\varphi$ the polynomial map from
$$\PG(k-1,q)\times \ldots \times \PG(k-1,q) \longrightarrow \PG(N^{k-1}-1,q),$$
where $N={{k+t-1}\choose{t}}$,
obtained as the composition of first applying the Veronese map 
$$\nu_{k,t}~:~\PG(k-1,q) \longrightarrow \PG(N-1,q),$$
in each of the $k-1$ factors, 
and then applying the Segre embedding
$$\sigma~:~\PG(N-1,q)\times \ldots \times \PG(N-1,q) \longrightarrow \PG(N^{k-1}-1,q).$$
Define $F$ as the polynomial map $\alpha \circ \varphi$.
It follows that $F$ is a homogeneous polynomial $F(Y_1,\ldots,Y_{k-1})$ where $Y_j=(Y_{j1},\ldots,Y_{jk})$,
which is homogeneous (of degree $t$) in each of the $Y_j$'s. Moreover, 
$$F(a_1,\ldots,a_{k-1})=(\alpha \circ \varphi)(a_1,\ldots,a_{k-1})=g(a_1,\ldots,a_{k-1})$$

For an ordered subset $S=[a_1,\ldots,a_{k-2}]$ of $\cA$, consider
$$
H(X):=F(a_1,\ldots,a_{k-2},X)-(-1)^{s(t+1)}f_S(X),
$$
a homogeneous polynomial of degree $t$.

The  polynomial $H(X)$ vanishes at the points of $\cA$ and therefore belongs to $\Phi_t[X]$, which proves (i).

For $S=[a_1,\ldots,a_{k-2}]$, where $a_i \in \cA$ and for which one of the $a_i$'s is repeated,
$$
F(a_1,\ldots,a_{k-2},a_{k-1})=g(a_1,\ldots,a_{k-2},a_{k-1})=0,
$$
which proves (ii).

To prove (iii) it suffices to prove that
$$
F(X_2,X_1,X_3,\ldots,X_{k-1})=(-1)^{t+1}F(X_1,X_2,X_3,\ldots,X_{k-1}) \pmod{\Phi_t[X_1],\ldots,\Phi_t[X_{k-1}]},
$$
the other transpositions following by the same argument.

By induction on $r$ we will prove that
$$
F(a_1,a_2,a_3,\ldots,a_r,X_{r+1},\ldots,X_{k-1})=(-1)^{t+1}F(a_2,a_1,a_3,\ldots,a_r,X_{r+1},\ldots,X_{k-1})$$
modulo $(\Phi_t[X_{r+1}],\ldots,\Phi_t[X_{k-1}])$.

This holds for $r=k-1$ (in which case there are no $X_i$'s) by Lemma~\ref{lem:g_function}.

By induction, whenever we evaluate at $X_r=a_r \in \cA$, the polynomial
$$
F(a_1,a_2,a_3,\ldots,a_{r-1},X_r,X_{r+1},\ldots,X_{k-1})-(-1)^{t+1}
F(a_2,a_1,a_3,\ldots,a_{r-1},X_r,X_{r+1},\ldots,X_{k-1})
$$
is zero modulo $(\Phi_t[X_{r+1}],\ldots,\Phi_t[X_{k-1}])$. Hence, it is zero modulo $(\Phi_t[X_{r}],\ldots,\Phi_t[X_{k-1}])$, which proves (iii).


To prove (iv), suppose that both $F$ and $G$ are polynomials satisfying (i), (ii) and (iii). Then 
$$
F(a_1,\ldots,a_{k-2},Y_{k-1})=G(a_1,\ldots,a_{k-2},Y_{k-1}) \pmod{\Phi_t[Y_{k-1}]}.
$$
for any $[ a_1,\ldots,a_{k-2}]$, where $a_j \in \cA$ are possibly repeated.

We proceed by induction. Suppose that for all $[ a_1,\ldots,a_{r}]$, where $a_j \in \cA$ are possibly repeated,
$$
F(a_1,\ldots,a_{r},Y_{r+1},\ldots,Y_{k-1})=G(a_1,\ldots,a_{r},Y_{r+1},\ldots,Y_{k-1})  \pmod{\Phi_t[Y_{r+1}],\ldots,\Phi_t[Y_{k-1}]}.
$$

Then, evaluating $Y_r$ at any point of $\cA$, the polynomial
$$
F(a_1,\ldots,a_{r-1},Y_{r},\ldots,Y_{k-1})-G(a_1,\ldots,a_{r-1},Y_{r},\ldots,Y_{k-1}),
$$
is zero $\pmod{\Phi_t[Y_{r+1}],\ldots,\Phi_t[Y_{k-1}]}$,
which implies that
$$
F(a_1,\ldots,a_{r-1},Y_{r},\ldots,Y_{k-1})=G(a_1,\ldots,a_{r-1},Y_{r},\ldots,Y_{k-1})  \pmod{\Phi_t[Y_{r}],\ldots,\Phi_t[Y_{k-1}]}.
$$
This complete the proof.
\qed

\begin{definition}
The multi-homogeneous polynomial $F(Y_1,\ldots,Y_{k-1})$ where $Y_j=(Y_{j1},\ldots,Y_{jk})$, which is homogeneous (of degree $t$) in each of the $Y_j$'s,
is called a {\em tensor form of $\cA$}. Note that a tensor form of an arc $\cA$ is unique modulo
$\Phi_t[Y_1],\ldots,\Phi_t[Y_{k-1}]$.
\end{definition}

\section{Hypersurfaces containing an arc}

Suppose $q=p^h$, where $p$ is an odd prime. Let $\cA$ be an arc in $\PG(k-1,q)$ of size $q+k-1-t$ and let $S$ be a subset of $\cA$ of size $k-3$. Projecting $\cA$ from $S$ one obtains a planar arc and the results from \cite{BaLa2018} apply. These results imply that $\cA$ is contained in a hypersurface of degree $t+p^{\lfloor \log_p t \rfloor}$, which is the cone of a planar curve of degree $t+p^{\lfloor \log_p t \rfloor}$ and the subspace $\langle S \rangle$.

The following theorem implies that there may be more hypersurfaces containing $\cA$. Indeed, we will consider a specific example in which Theorem~\ref{morehyps} tells us more than what we obtain from simply projecting.

In the following theorem, $X^{(i_{1},\ldots,i_{k})}=X_1^{i_1}\cdots X_k^{i_k}$. 

We do not claim in Theorem~\ref{morehyps} that there are coefficients which are not zero. In \cite{BaLa2018} it was proven that if $k=3$ and $q$ is odd then there are non-zero coefficients and therefore low-degree polynomials which are zero on $\cA$. We will prove that there are non-zero coefficients for $q$ odd when $t=2$ and $k=4$ in Theorem~\ref{therestwoquadrics}.


\begin{theorem} \label{morehyps}
Let $\cA$ be an arc in $\PG(k-1,q)$ of size $q+k-1-t$ and let $F(Y_1,\ldots,Y_{k-1})$ be the $(t,t,\ldots, t)$-form whose existence is given by Theorem~\ref{thm:form_existence}. 
If $\cA$ is not contained in a hypersurface of degree $t$ then the coefficient of $Y_1^{i_1}\cdots Y_{k-2}^{i_{k-2}}$, where $i_m=(i_{m1},\ldots,i_{mk})$, in
$$
F(Y_1+X,\ldots,Y_{k-2}+X,X)-F(Y_1,\ldots,Y_{k-2},X)
$$
is either zero or a homogeneous polynomial in $X$ of degree
$$
(k-1)t-\sum_{m=1}^{k-2} \sum_{j=1}^k i_{mj},
$$
which is zero on $\cA$.
\end{theorem}

\begin{proof}
Let $x \in \cA$ and let $F_x(Y_1,\ldots,Y_{k-2})$ be a tensor form of the 
projection $\overline \cA$ of $\cA$ from $x$ (obtained by applying Theorem~\ref{thm:form_existence}). Explicitly, this can be done in the following way. 

Choose a coordinate $j$ such that $x_j \neq 0$. For each $a=(a_1,\ldots,a_k)\in \cA$, define a point $\overline a$ of $\PG(k-2,q)$, whose $i$-th coordinate is $a_ix_j-a_jx_i$, for $i\neq j$. 
So $\overline a$ has no $j$-th coordinate.

Let 
$$
\overline \cA=\{ \overline a \ | \ a \in \cA \setminus \{ x\} \}.
$$
Theorem~\ref{thm:form_existence} implies the existence of a form $G(Z_1,\ldots,Z_{k-2})$ for $\overline \cA$. 
Note that each $Z_{mi}$ has no $j$-th coordinate.
Then define $F_x$ as the polynomial obtained from $G$ by 
substituting $Z_{mi}=x_jY_{mi}-x_iY_{mj}$ for $i\neq j$.

Since $Z_{mi}$ is unaffected by the substitution $Y_{mi} \mapsto Y_{mi}+x_i$
$$
F_x(Y_1+x,\ldots,Y_{k-2}+x)=F_x(Y_1,\ldots,Y_{k-2}).
$$

Both $F_x(Y_1,\ldots,Y_{k-2})$ and $F(Y_1,\ldots,Y_{k-2},x)$ satisfy all the properties of the tensor form obtained by applying Theorem~\ref{thm:form_existence} to the arc $\overline \cA$, apart from the fact that each $Y_j$ is a $k$-tuple and not a $(k-1)$-tuple. However, the same uniqueness argument used in part (iv) of the proof of Theorem 1 applies, so they are the same.

Therefore, for all $x \in \cA$,
$$
F(Y_1+x,\ldots,Y_{k-2}+x,x)=F(Y_1,\ldots,Y_{k-2},x),
$$
from which the theorem follows.
\end{proof}

We now consider an example which illustrates that Theorem~\ref{morehyps} can prove the existence of hypersurfaces containing $\cA$ which are not obtained simply by projection.

\begin{theorem} \label{therestwoquadrics}
If $q$ is odd then an arc of size $q+1$ in $\mathrm{PG}(3,q)$ is contained in a quadric.
\end{theorem}

\begin{proof}
Suppose that $\cA$ is an arc of $\mathrm{PG}(3,q)$ of size $q+1$ not contained in a quadric. Then $k=4$, $t=2$ and $\Phi_t=\{ 0 \}$.

By Theorem~\ref{thm:form_existence}, there is a $(2,2,2)$-form
$$
F(Y_1,Y_2,Y_3)=\sum_{j_1,j_2,j_3} b_{j_1,j_2,j_3} Y_1^{j_1} Y_2^{j_2} Y_3^{j_3}, 
$$
where the sum runs over all $j_m=(j_{m1},j_{m2},j_{m3},j_{m4})$ such that $j_{m1}+j_{m2}+j_{m3}+j_{m4}=2$, with the properties therein stated.

Since $t+1$ is odd, and $\Phi_t=\{ 0 \}$, Theorem~\ref{thm:form_existence} (iii) implies
$$
b_{j_1,j_2,j_3}=-b_{j_1,j_3,j_2}.
$$
Since $F_{e_1}(Y_1,Y_2)$ has no $Y_{11}$ or $Y_{21}$ terms
$$
b_{(2,0,0,0),j_2,j_3}=0,
$$ 
if $j_{11}\neq0$ or $j_{21}\neq0$.

Applying Theorem~\ref{morehyps} to the coefficient $Y_1^{(2,0,0,0)}Y_2^{(0,1,0,0)}$, we have that the polynomial
\begin{equation} \label{hyp1}
g:=2X_2f_{e_1e_2}(X_3,X_4)+(b_{(2,0,0,0),(0,1,1,0),(0,0,1,1)}+b_{(2,0,0,0),(0,1,0,1),(0,0,2,0)})X_3^2X_4
\end{equation}
$$
+(b_{(2,0,0,0),(0,1,0,1),(0,0,1,1)}+b_{(2,0,0,0),(0,1,1,0),(0,0,0,2)})X_3X_4^2
$$
defines a hypersurface containing $\cA$. 

Note that it is not zero, since $q$ is odd and 
$f_{e_1e_2}(X_3,X_4) \neq 0$.

Applying Theorem~\ref{morehyps} to the coefficient $Y_1^{(0,2,0,0)}Y_2^{(1,0,0,0)}$, we have that the polynomial
\begin{equation} \label{hyp2}
h:=2X_1f_{e_1e_2}(X_3,X_4)+(b_{(0,2,0,0),(1,0,1,0),(0,0,1,1)}+b_{(0,2,0,0),(1,0,0,1),(0,0,2,0)})X_3^2X_4
\end{equation}
$$+(b_{(0,2,0,0),(1,0,0,1),(0,0,1,1)}+b_{(0,2,0,0),(1,0,1,0),(0,0,0,2)})X_3X_4^2
$$
defines a hypersurface containing $\cA$. 

Then dividing $X_1g-X_2h$ by $X_3X_4$ we have that there is a polynomial
$$
c_{13}X_1X_3+c_{14}X_1X_4+c_{23}X_2X_3+c_{24}X_2X_4,
$$
which is zero on $\cA$. Again, this polynomial is not zero since this would imply that $2X_2f_{e_1e_2}(X_3,X_4)$ is zero on $\cA$, which it is not.

Hence, $\cA$ is contained in a quadric.
\end{proof}

\section{The Segre-Blokhuis-Bruen-Thas hypersurface} \label{BBThyp} \label{evensection}

In this section we elaborate on the hypersurface associated to an arc of hyperplanes in $\PG(k-1,q)$ obtained in \cite{Segre1967} for $k=3$, in \cite{BBT1988} for $k=4,5$, and \cite{BBT1990} for arbitrary dimension $k\geq 3$.
We will give a new proof for its existence and compare this result with Theorem 1.

For $j=1,\ldots,k-1$ consider $X_j=(X_{j1},\ldots,X_{jk})$ as a $k$-tuple of indeterminates.
We denote by 
$$\det_i(X_1,\ldots,X_{k-1})$$ 
the determinant of the matrix which is obtained from the matrix with the $X_j$'s as rows and the $i$-th column deleted.

The main theorem of \cite{BBT1990} implies that there is a homogeneous polynomial $\phi(Z_1,\ldots,Z_k)$ of degree $t$ for $q$ even and of degree $2t$ for $q$ odd, which vanishes at the points of the dual space which are dual to the hyperplanes containing exactly $k-2$ points of an arc $\cA$. We paraphrase the main result of \cite{BBT1990} as follows.

\begin{theorem} \label{thm:SBBT}
Let $m\in \{1,2 \}$ be such that $m-1 \equiv q$ mod $2$. If $\cA$ is an arc in $\PG(k-1,q)$ of size $q+k-1-t$, where $|\cA| \geq mt+k-1$, then there is a homogeneous polynomial in $k$ variables $\phi(Z_1,\ldots,Z_k)$, of degree $mt$, which gives a polynomial $G(X_1,\ldots,X_{k-1})$ in $k(k-1)$ indeterminates under the substitution $Z_j=\det_j(X_1,\ldots,X_{k-1})$, with the property that for each $(k-2)$-subset $S=\{ y_1,\ldots,y_{k-2} \}$ of $\cA$
$$
G(y_1,\ldots,y_{k-2},X)=(f_S(X))^m.
$$ 
\end{theorem}
\begin{proof}
Order the arc $\cA$ arbitrarily and let $E$ be a subset of $\cA$ of size $mt+k-1$. Define
\begin{eqnarray}\label{eqn:G}
G(X_1,\ldots,X_{k-1})=\sum_{T} \left ( f_{T\setminus \{a_{k-1}\}}(a_{k-1})\right )^m \prod_{u \in E \setminus T} \frac{\det(X_1,\ldots,X_{k-1},u)}{\det(a_1,\ldots,a_{k-1},u)}.
\end{eqnarray}
where the sum runs over subsets $T=\{a_1,\ldots,a_{k-1}\}$ of $E$.

Observe that $G$ can be obtained from a homogeneous polynomial of degree $mt$ in $Z_1,\ldots,Z_k$ under the change of variables
$Z_j=\det_j(X_1,\ldots,X_{k-1})$.

For $S=\{y_1,\ldots,y_{k-2}\}$ define
$$h_S(X):=G(y_1,\ldots,y_{k-2},X).$$

Note that $h_S(X)$ is well-defined since any reordering of $S$ can only ever multiply $h_S(X)$ by $(-1)^{mt}=1$.

For $S\subset E$, the only nonzero terms in $h_S(X)$ are obtained for subsets $T$ of $E$ containing
$S$. Therefore
$$
h_S(X)=\sum_{a\in E\setminus S} \left ( f_{S}(a)\right )^m \prod_{u \in E \setminus (S\cup\{a\})} \frac{\det(y_1\ldots,y_{k-2},X,u)}{\det(y_1,\ldots,y_{k-2},a,u)}.
$$
The evaluation of $h_S(X)$ at $x\in E$ is equal to zero if $x\in S$ and equal to $(f_S(x))^m\neq 0$ otherwise. Since, with respect to a basis containing $S$ both $f_S^m$ and $h_S$ are homogeneous polynomials in two variables of degree $mt$, we conclude that $h_S=f_S^m$.

If $S$ is not contained in $E$ then we proceed by induction on the size of $S\setminus E$. As induction hypothesis we assume that for each subset $S$ with $S\setminus E$ of size $r$ the polynomials $h_S$ and $f_S^m$ are equal. Let $S=\{y_1,\ldots,y_{k-2}\}$ be such that $S\setminus E$ is of size $r+1$. W.l.o.g. assume $y_{k-1}\notin E$. Then for $x\in E$ we have
$$
h_S(x)=(-1)^{mt}h_{S'}(y_{k-1})=h_{S'}(y_{k-1}),
$$
where $S'$ is the set obtained from $S$ by replacing the $(k-1)$-th element $y_{k-1}$ of $S$ by $x$.
On the other hand, by the definition (\ref{eqn:g}) of $g$ and the scaled coordinate-free lemma of tangents, we have
$$
(f_S(x))^m=(g(y_1,\ldots,y_{k-1},x))^m=(g(y_1,\ldots ,y_{k-2},x,y_{k-1}))^m=(f_{S'}(y_{k-1}))^m.
$$
By induction $h_{S'}(y_{k-1})=(f_{S'}(y_{k-1}))^m$, and therefore the polynomials $h_S$ and $f_S^m$ have the same evaluation at points in $E$. Applying the same argument as in the case where $S\subset E$ we obtain $h_S=f_S^m$.
\end{proof}

We now compare Theorem \ref{thm:SBBT} to Theorem~\ref{thm:form_existence}. First, observe that the polynomial $G$ as defined in (\ref{eqn:G}) is homogeneous of degree $mt$ in each of its $k$-tuples of variables, and $G$ takes the value zero when evaluated at an argument which contains repeated points. Next, by Theorem \ref{thm:SBBT}, for any subset $S=\{a_1,\ldots,a_{k-2}\}$ of $\cA$ we have $G(a_1,\ldots,a_{k-1},X)=(f_S(X))^m$.
Also, as we already explained in the proof of Theorem \ref{thm:SBBT}, it follows from the scaled coordinate-free lemma of tangents that the polynomial $G$ is not affected by reordering of the points in its arguments. We obtain the following theorem.

\begin{theorem} \label{thm:q_even_odd}
Let $m\in \{1,2 \}$ be such that $m-1 \equiv q$ mod $2$. If $\cA$ is an arc in $\PG(k-1,q)$ of size $q+k-1-t$, 
where $|\cA| \geq mt+k-1$, then there exists a polynomial $G(Y_1,\ldots,Y_{k-1})$ (in $k(k-1)$ variables) which is homogeneous of degree $mt$ in each of the $k$-tuples of variables $Y_j$, with the following properties.
\begin{enumerate}
\item[(i)]
$G(a_1,\ldots,a_{k-2},X)=(f_S(X))^m$ for every $(k-2)$-subset $S=\{a_1,\ldots,a_{k-2}\}$ of $\cA$;
\item[(ii)]
$G(a_1,\ldots,a_{k-1})=0$ if $a_i=a_j$ for some $i\neq j$;
\item[(iii)]
$G$ is symmetric in its $k-1$ arguments $Y_1,\ldots,Y_{k-1}$;
\end{enumerate}
\end{theorem}

Note that for $q$ even, Theorem \ref{thm:q_even_odd} is an improvement of Theorem~\ref{thm:form_existence}. It proves that the modulo $\Phi_t[X]$ condition is not necessary in Theorem~\ref{thm:form_existence} for $q$ even, although the uniqueness would not be valid without the modulo $\Phi_t[X]$ condition. For $q$ odd, Theorem \ref{thm:q_even_odd} has the advantage that its properties hold true without the modulo $\Phi_t[X]$ condition; the disadvantage is that the degree of $G$ in each of its $k$-tuples of arguments is $2t$ whereas for the form $F$ from Theorem~\ref{thm:form_existence} it is only $t$. We do not believe that the modulo $\Phi_t[X]$ condition is necessary in Theorem~\ref{thm:form_existence} for $q$ odd although, as in the $q$ even case, the uniqueness would not be valid without the modulo $\Phi_t[X]$ condition.


  \bigskip
  
   Simeon Ball,\\
   Departament de Matem\`atiques, \\
Universitat Polit\`ecnica de Catalunya, \\
M\`odul C3, Campus Nord,\\
c/ Jordi Girona 1-3,\\
08034 Barcelona, Spain \\
   {\tt simeon@ma4.upc.edu}
  
  \bigskip
  
Michel Lavrauw,\\
Faculty of Engineering and Natural Sciences,\\
Sabanc\i \  University,\\
Istanbul, Turkey\\
{\tt mlavrauw@sabanciuniv.edu}

\end{document}